\documentclass[12pt,a4paper,reqno]{amsart}
\usepackage{epsfig}
\usepackage{mathabx}
\usepackage[left=2.7cm, right=2.7cm, top=2.6cm, bottom=2.6cm]{geometry}

\def\be{\begin{equation}}
\def\ee{\end{equation}}
\def\bea{\begin{eqnarray}}
\def\eea{\end{eqnarray}}
\def\bes{\begin{eqnarray*}}
\def\ees{\end{eqnarray*}}

\def\nn{\nonumber}
\def\lb{\label}

\def\R{{\bf R}}
\def\C{{\bf C}}
\def\Z{{\bf Z}}

\def\N{{\bf N}}

\def\Q{{\bf Q}}

\def\aa{{\alpha}}
\def\bb{{\beta}}
\def\ga{{\gamma}}

\def\th{{\theta}}

\def\om{{\omega}}
\def\Om{{\Omega}}

\def\sg{{\sigma}}
\def\dm{{\diamond}}

\def\vf{{\varphi}}

\def\<{{\langle}}
\def\>{{\rangle}}

\def\im{{\rm im}}

\def\Sp{{\rm Sp}}

\def\ol{\overline}

\def\hb{\vrule height0.18cm width0.14cm $\,$}

\newtheorem{thm}{Theorem}[section]

\newtheorem{lem}[thm]{Lemma}
\newtheorem{pro}[thm]{Proposition}

\theoremstyle{definition}
\newtheorem{rem}[thm]{Remark}

\newtheorem{exa}[thm]{Example}

\makeatletter
\@namedef{subjclassname@2020}{\textup{2020} Mathematics Subject Classification}
\makeatother

\title[Stability of closed characteristics and invariant sets]{Stability of closed characteristics and invariant sets on star-shaped hypersurfaces}
\author[Huagui Duan]{Huagui Duan}
\thanks{Huagui Duan was partially supported by NNSFC (Nos. 12271268 and 12361141812) and Natural Science Foundation
of Tianjin (No. 25JCZDJC01030). Zihao Qi was partially supported by NNSFC (Nos. 12271268 and 125B200041).}
\address{Huagui Duan, School of Mathematics, Jilin University, Changchun, Jinlin, 130012, The People's Republic of China}

\email{duanhg@nankai.edu.cn.}

\author[Zihao Qi]{Zihao Qi}
\address{Zihao Qi, School of Mathematical Sciences, Nankai University, Tianjin 300071, The People's Republic of China}
\email{1120230029@mail.nankai.edu.cn.}
\date{2026-08-02}

\begin{document}
\maketitle

\begin{abstract}
{\it  This paper focuses on the stability of closed characteristics and their nearby invariant sets on compact star-shaped hypersurfaces in $\R^{2n}$ with finitely many simple closed characteristics. Firstly, it is proved that all closed characteristics are non-hyperbolic, under some index condition weaker than dynamical convexity. Secondly, when $n=3$, it is proved that all closed characteristics are elliptic when their number is exactly $3$, under non-degeneracy and some minor index condition. Lastly, it is proved that all simple closed characteristics are either degenerate at some iteration or not locally maximal under dynamical convexity. }
\end{abstract}

\renewcommand{\theequation}{\thesection.\arabic{equation}}
\renewcommand{\thefigure}{\thesection.\arabic{figure}}

\setcounter{figure}{0}
\setcounter{equation}{0}
\section{Introduction and main result}
For a contact manifold $M$ equipped with a contact form $\alpha$, one can associate the {\it Reeb vector field} $X$, satisfying $ \alpha(X)=1, d\alpha(X,\cdot)=0 $. A closed Reeb orbit means the orbit of $X$ that closes up. In this paper, we mainly consider contractible closed Reeb orbits. Such an orbit is called {\it simple} if it is not a multiple covering (i.e., iteration) of any other contractible closed Reeb orbit. Here, the $m$-th iteration of a closed Reeb orbit $x:\R/T\Z\to M$ is defined by $x^m:\R/mT\Z\to M(t\mapsto x(t))$. Two simple closed Reeb orbits $x$ and $y$ are considered {\it distinct} if there does not exist $\th\in (0,T)$ such that $x(t)=y(t+\th)$ for all $t\in\R$. We shall omit the word {\it distinct} when we talk about more than one simple closed Reeb orbit. A closed Reeb orbit is {\it non-degenerate} if $1$ is not the eigenvalue of its linearized Poincar\'{e} map. A simple closed Reeb orbit is {\it strongly non-degenerate} if all its iterations are non-degenerate. A contact form is {\it non-degenerate} if all its closed Reeb orbits are non-degenerate.

In this paper, we mainly focus on starshaped hypersurfaces $\Sigma$ in $\R^{2n}$ endowed with the standard contact form. It is well known that it is strictly contactomorphic to the standard $(2n-1)$ dimensional contact sphere. We will call its closed Reeb orbit a {\it closed characteristic} below. A longstanding conjecture called multiplicity conjecture asserts that for star-shaped hypersurfaces in $\R^{2n}$, the number of simple closed characteristics is at least $n$. This conjecture is widely open without any restriction except in $\R^4$, see \cite{CHHL24}. Some progress are made given some non-degeneracy or index conditions such as dynamical convexity, cf. \cite{CGG25,CHHL24,DLLW24,DLLQW25,DX25} for some recent progress. In \cite{HWZ03}, Hofer, Wysocki and Zehnder proved for a broad class of star-shaped hypersurfaces in $\R^4$ that the number of simple closed characteristics is either two or infinite and conjectured this is true for every star-shaped hypersurface in $\R^4$.
The so-called {\it $n$-or-infinity conjecture} states that the number of simple closed characteristics is either $n$ or infinite for every star-shaped hypersurface in $\R^{2n}$, cf. \cite{CGG25,CHHL24,DLR22} for recent progress. The well-known example supporting this conjecture is the irrational ellipsoid, which has exactly $n$ simple closed characteristics.

A closed Reeb orbit is {\it hyperbolic} if no eigenvalue of its linearized Poincar\'{e} map is on the unit circle, {\it elliptic} if all eigenvalues of its linearized Poincar\'{e} map are on the unit circle and {\it irrationally elliptic} if its linearized Poincar\'{e} map can be represented by the direct sum of $2\times 2$ irrational rotations in a symplectic frame. A contact form $\alpha$ on $S^{2n-1}$ is called {\it dynamically convex} if every closed Reeb orbit has Conley-Zehnder index at least $n+1$.

In \cite{WHL07}, Wang, Hu and Long proved that all closed characteristics on a strictly convex hypersurface in $\R^{4}$ are irrationally elliptic when there are only finitely many simple closed characteristics, and conjectured that this is true in $\R^{2n}$. Note that all closed characteristics on irrational ellipsoids are irrationally elliptic. In \cite{CHHL24}, the same result still holds for star-shaped hypersurfaces in $\R^{4}$, so one tends to believe that similar phenomenon happens for star-shaped hypersurfaces in $\R^{2n\ge6}$, with the knowledge that it happens for all known examples of star-shaped hypersurfaces with finitely many simple closed characteristics. Recently, it is proved in \cite{LLW25} that, under dynamical convexity, there are at least two irrationally elliptic simple closed characteristics when there are exactly $n$ simple closed characteristics. Actually, the assumption can be relaxed to that there are finitely many simple closed characteristics.

If we only detect non-hyperbolic closed characteristics, the stronger conclusions can be obtained. For example, it was shown in \cite{CGGM26} that, under possibly weaker index conditions than dynamical convexity, all closed characteristics are non-hyperbolic when there are finitely many simple closed characteristics. In fact, this index condition requires different lower bounds for each closed Reeb orbit $x$ according to the $1$-algebraic multiplicity of its linearized Poincar\'{e} map, denoted by $m_1(x)$. Specially, when there is a totally degenerate closed Reeb orbit, i.e., all eigenvalues of its linearized Poincar\'{e} map are equal to $1$, the index assumption is being no less than $n+1$. Our first result is to show Theorem A of \cite{CGGM26} still holds under the more weaker index condition.

\begin{thm}\label{weak condition non-hyperbolic}
Assume that a compact star-shaped hypersurface in $\R^{2n\ge6}$ has only finitely many simple closed characteristics. If every closed characteristic $x$ with positive mean index has (lower) Conley-Zehnder index no less than $\min\{\max\{3,2+\frac{m_1(x)}{2}\},n\}$, then all of them are non-hyperbolic. In particular, the result holds for all closed characteristics if their underlying simple closed orbits have index no less than $n$.
\end{thm}
\begin{rem}\label{index-notation}
The Conley-Zehnder index is defined by viewing a closed characteristic as a closed Reeb orbit on the star-shaped hypersurface rather than a Hamiltonian periodic orbit in $\R^{2n}$.
\end{rem}
A star-shaped hypersurface with finitely many simple closed characteristics is also called a pseudo-rotation. Due to the $n$-or-infinity conjecture, for the study of dynamical behaviors of a pseudo-rotation, it is valuable to study star-shaped hypersurfaces in $\R^{2n}$ with exactly $n$ simple closed characteristics. For such special star-shaped hypersurface, it is obtained in \cite{CGG25} that all closed characteristics are irrationally elliptic under non-degeneracy, dynamical convexity and some non-resonance condition, although such non-resonance condition does not hold for every irrational ellipsoids, see the example in Section \ref{Proof}. Since $3$-dimensional case is well solved as mentioned before, we will drop the non-resonance condition and largely weaken the index condition to obtain that all closed characteristics are elliptic for star-shaped hypersurfaces in $\R^6$.

In \cite{DLLQW25}, multiplicity conjecture on star-shaped hypersurfaces in $\R^6$ has been proved under non-degeneracy and some very minor index conditions such as that no simple closed Reeb orbit has Conley-Zehnder index $1$. Note that the index there equals to the index minus one in this paper, see Remark \ref{index-notation}. Under the same condition, we will analyse the ellipticity of closed characteristics and invariant sets near them when there are exactly $3$ simple closed characteristics. Here a closed Reeb orbit is {\it locally maximal}, if it has a neighborhood in which there exists no other invariant set. A typical example is the hyperbolic orbit.

\begin{thm}\label{3elliptic}
Assume that a compact star-shaped hypersurface in $\R^6$ is non-degenerate and no simple closed characteristic has Conley-Zehnder index $1$, then all of them are elliptic and not locally maximal, when there are exactly $3$ simple closed characteristics. Moreover, two of them are irrationally elliptic.
\end{thm}

To the authors' knowledge, it is the first time to obtain more than two elliptic closed orbits on pseudo rotations of contact spheres under some natural condition satisfied by irrational ellipsoids.

\begin{rem}
When $n\ge3$, irrational ellipticity of a closed characteristic $x_i$ is usually detected by checking if the index of its $2m_i$-th iteration equals to $2N\pm(n-1)$ as in \cite{DQ26,HW22,LLW25,W22}, where $N$ and $m_i$ are chosen by Common Index Jump Theorem. However, in this way, we may never detect more than $2$ irrationally elliptic closed characteristics under certain conditions satisfied by irrational ellipsoids, see Example 4.3 in Section \ref{Proof}.
\end{rem}
\begin{rem}
On star-shaped hypersurfaces in $\R^{2n}$, under non-degeneracy and the index condition as in \cite{DLLW24} non-hyperbolicity can be confirmed, or more generally, the number of the $2\times2$ rotation matrix can be estimated. It follows from \cite{DLLW24} and Theorem A of \cite{CGGM26} immediately that all closed characteristics are non-hyperbolic when there are exactly $n$ simple closed characteristics. Moreover, for star-shaped hypersurfaces in $\R^8$, the index condition can be further weakened to that in \cite{DX25}.
\end{rem}

The main new ingredient of the proof of Theorem \ref{3elliptic} is Theorem B of \cite{CGGM26} which confirms the non-local-maximality of all closed characteristics under dynamical convexity and non-degeneracy. We will exclude other cases of its index conditions than dynamical convexity by some index analysis.

Now we consider Theorem B of \cite{CGGM26} further and remove the non-degeneracy assumption in some sense.

\begin{thm}\label{isoinvar}
Assume that a compact star-shaped hypersurface in $\R^{2n\ge4}$ has only finitely many simple closed characteristics and every closed characteristic with positive mean index has (lower) Conley-Zehnder index no less than $n+1$, then no simple closed characteristic is strongly non-degenerate and locally maximal.
\end{thm}

Thus notice that Theorem B of \cite{CGGM26} actually forbids the appearance of (strongly non-degenerate and) locally maximal closed Reeb orbit under the additional assumption that $\alpha$ is non-degenerate.

The proofs of Theorem \ref{weak condition non-hyperbolic} and Theorem \ref{isoinvar} are arguments by contradiction, following the strategy of \cite{CGGM26} to look at the action or index distance from the special kind of closed characteristics that we hope to exclude. In the proof of Theorem \ref{weak condition non-hyperbolic}, the distance between the index of the special closed characteristic at some iterations and the support of the local Floer homology of another closed characteristic will be estimated in some case that two kind of difference above can not provide useful information, which is also enough to get a contradiction. In addition, a symmetric argument in the Common Index Jump Theorem will be applied to narrow the range of the index of the special closed characteristic such that its distance to the larger range of indices of other closed characteristics caused by degeneracy can be still effectively estimated.

\setcounter{equation}{0}

\section{Floer homology}
We adopt the framework from \cite{BO09, CGGM26} and references therein for details. Let $\alpha$ be the contact form on the boundary $M=\partial W$ of a Liouville domain $W^{2n\geq 4}$. 
Assume that \be c_1(TW)|_{\pi_2(W)}=0.\lb{c_1=0}\ee Denote by $\widehat{W}$ the symplectic completion of $W$, i.e., $\widehat{W}=W\cup_{M}M\times[1,\infty)$
with the symplectic form $\omega$ extended to $M\times[1,\infty)$ as $d(r\alpha)$, where $r$ is the coordinate on $[1,\infty)$. A classical example is the star-shaped domain in $(\R^{2n}, \omega_{std})$, whose boundary is contactomorphic to the standard contact sphere $M=S^{2n-1}$.

We will concentrate on closed Reeb orbits of $(M,\alpha)$ which are contractible in $W$ and denote the collection of such orbits by $\mathcal{P}(\alpha)$. The collection of periods of all orbits in $\mathcal{P}(\alpha)$ is denoted by $\mathcal{S}(\alpha)$. For each closed Reeb orbit $x\in\mathcal{P}(\alpha)$, let $P_x\in \Sp(2n-2)$ be its linearized Poincar\'{e} map. We denote $\nu(x)$ by the nullity of $P_x$ and, due to (\ref{c_1=0}), $\hat{\mu}(x), \mu_+(x), \mu_-(x)$ by the mean, upper and lower Conley-Zehnder index of the linearized flow along it under the symplectic trivialization of the contact structure induced by the disk capping of $x$, see Section 3 for details. When $x$ is non-degenerate, the upper and lower Conley-Zehnder index coincide. At this time, we call them the Conley-Zehnder index of $x$ and denote $\mu(x)$, omitting the subscript. For Hamiltonian $1$-periodic orbits on $\widehat{W}$, these indices are defined similarly for the linearized Hamiltonian flow along it, which is a path in $\Sp(2n)$.

The Hamiltonian action functional $\mathcal{A}_H: C^\infty(S^1,W)\to \R$ is defined by $$\mathcal{A}_H(x)=\int_{\bar{x}}\omega-\int_{S^1}H(x(t))dt,$$
where $\bar{x}$ is a disk bounded by $x$. Its critical points are exactly $1$-periodic orbits of $H$ and the set of critical values is denoted by $\mathcal{S}(H)$.
\subsection{Floer complex counting cascades}
Below we consider the set $\mathcal{H}$ of Hamiltonians on $\widehat{W}$ depending only on $r$ outside $W$, i.e., $H=h(r)$ on $M\times[1,\infty)$ and satisfying
\begin{itemize}
    \item $H=0$ on $W$ and $h$ is monotone increasing;
    \item $h''>0$ on $(1,r_{\max})$ for some $r_{\max}>1$ depending on $h$;
    \item $h(r)=ar-c$, $a\notin\mathcal{S}(\alpha)$ when $r\geq r_{\max}$.
\end{itemize}
The Hamiltonian vector field $X_H$ is determined by $\omega(X_H,\cdot) = -dH$. For $H\in\mathcal{H}$,
$ X_H = h'(r)R_\alpha$ on $M\times[1,\infty)$, where $R_\alpha$ is the Reeb vector field. Hence $T$-periodic orbits $x$ of the Reeb flow with $T<a$ are in one-to-one correspondence with $1$-periodic orbits $\tilde{x}=(x,r)$ of $H$ with $ h'(r)=T$. Late we call $a$ the {\it slope} of $H$. Note that there is a natural partial order in $\mathcal{H}$ defined by $H_0\preceq H_1\Leftrightarrow H_0(\cdot)\le H_1(\cdot)$ at any point in $\widehat{W}$.

Now we introduce the Floer homology of $H\in\mathcal{H}$. It is the homology of the Floer complex $\mathrm{CF}(H)$ defined here in a somewhat unusual way by applying the standard Morse-Bott complex construction using ``cascades'' to the non-constant $1$-periodic orbits of $H$, see \cite{BO09}. The coefficient of the complex will be taken as $\Q$.

Next, we give more description about $\mathrm{CF}(H)$ that we will use in the proof. We first consider the case where $\alpha$ is non-degenerate. In this case, the $1$-periodic orbit corresponding to $x$, $\tilde{x}$, is Morse-Bott non-degenerate. After fixing a perfect Morse function $f_x$ and a Riemannian metric on $\tilde{x}(S^1)$, we can perturb $H$ to a non-degenerate time-dependent Hamiltonian $H_1$ (cf. (25) in \cite{BO09}) whose $1$-periodic orbits consist of
\begin{itemize}
   \item Constant orbits $p$ in $W$, corresponding to the (non-degenerate) critical point of $H_1$;
   \item Reparameterizations of $\tilde{x},\ \hat{x}\ {\rm and}\ \check{x}$, corresponding to the maximum and the minimum of $f_x$.
\end{itemize}
Then the Floer complex is defined as
\be {\rm CF}(H)=\mathop{\oplus}\limits_{x\in\mathcal{P}(\alpha)}\Q\langle\hat{x},\check{x}\rangle\ \oplus\mathop{\oplus}\limits_{p\in {\rm crit}(H_1)}\Q\langle p\rangle,\ee
graded by the Conley-Zehnder index. Specifically, the degree of the generators are \be\label{generator degree}|\hat{x}| = \mu(x) + 1,\ |\check{x}| = \mu(x),\ |p|=n-\mu_{Morse}(p).\ee  

After choosing a regular almost complex structure $J$, the differential of this chain complex $\partial$ is defined by \be \partial\hat{x}= \sum\limits_{|\hat{x}|-|\hat{y}|=1}n(\hat{x},\hat{y})\langle\hat{y}\rangle
+\sum\limits_{|\hat{x}|-|\check{y}|=1}n(\hat{x},\check{y})\langle\check{y}\rangle
+\sum\limits_{|\hat{x}|-|p|=1}n(\hat{x},p)\langle p\rangle, \ee where $n(\hat{x},\hat{y})$, etc count cascades from $\hat{x}$ to $\hat{y}$ with signs determined by a coherent orientation (see (36) in \cite{BO09}). Roughly speaking, a cascade is the concatenation of integral curves of $-\nabla f_x$ and Floer trajectories with respect to $H$ (c.f. (39)-(40) in \cite{BO09}). $\partial\check{x}$ is defined in a similar fashion. Note that $n(\hat{x},\check{x})=0$ since the Morse function $f_x$ is perfect.


When $\alpha$ is degenerate, the Floer complex of $H$ is defined as above by replacing it with a $C^\infty$-small Morse-Bott perturbation $\tilde{H}$. If the closed Reeb orbits of $\alpha$ with period less than $\operatorname{slope}(H)$ are isolated, then so are the non-constant $1$-periodic orbits of $H$ and they split into all (Morse-Bott non-degenerate and) non-constant closed orbits of the perturbation $\tilde{H}$. We refer the reader to Section 3.3.2 of \cite{CGGM26} for further details.
\subsection{Filtered and local homology}\label{Filtered and local homology}
Actually, the complex $\mathrm{CF}(H)$ is filtered by the Hamiltonian action $\mathcal{A}_H$
 and the homology of the filtered complex $\mathrm{CF}^I(H)$ is still the usual filtered Floer homology $\mathrm{HF}^I(H)$,
 which fits into the long exact sequence, \be\label{long exact sequence}
\cdots \to \mathrm{HF}_m^{(a,b)}(H) \to \mathrm{HF}_m^{(a,c)}(H) \to \mathrm{HF}_m^{(b,c)}(H) \to \mathrm{HF}_{m-1}^{(a,b)}(H) \to \cdots,
\ee where $a,b,c\notin\mathcal{S}(H)$ and $-\infty\le a<b<c\le+\infty$.  For $H_0\preceq H_1$, there is a continuation map,$$\mathrm{HF}^I(H_0)\to\mathrm{HF}^I(H_1),$$where $I$ is disjoint from $\mathcal{S}(H_0)\cup\mathcal{S}(H_1)$, $\mathcal{S}(H)$ is the set consisting of Hamiltonian actions of every $1$-periodic orbit. Then symplectic homology is defined by taking the direct limit, $$\mathrm{SH}^I(W)=\lim_{\mathop{\longrightarrow}\limits_{H\in\mathcal{H}}}\mathrm{HF}^I(H).$$Its equivariant version $\mathrm{SH}^{S^1,I}(W)$ can be defined similarly, starting from equivariant Floer homology, see \cite{GG20}. It is well-known when $W$ is a star-shaped domain that
\be\label{contact sphere homology}\mathrm{SH}_*^{S^1,(0,\infty)}(W)=\left\{\begin{array}{ll} \Q,\quad \mbox{if}\ *=n+1,n+3,\cdots\\
           0, \quad\mbox{otherwise}  \end{array}\right.\ee

For an isolated closed Reeb orbit of $\alpha$ with period $T<\operatorname{slope}(H)$, the $1$-periodic orbit $\tilde{x}$ of $H$ is also isolated as a $1$-periodic orbit of $X_H$. As a consequence, we obtain an isolated circle $\Gamma=\tilde{x}(S^1)$ of fixed points for the time-$1$ map of the Hamiltonian flow. We denote by $\mathrm{HF}(\tilde{x})$ the local Floer homology of $\Gamma$. The homology group of this type was first considered in \cite{F89a,F89b} and we refer the reader to \cite{G10} for a more similar context. It is well known that
\be\label{support range}\mathrm{supp}\mathrm{HF}(\tilde{x}) \subset [\mu_-(x),\mu_+(x)+1], \ee
where the support $\mathrm{supp}\mathrm{HF}(\tilde{x})$ is denoted by $\{k\in\Z|\mathrm{HF}_k(\tilde{x})\neq0\}$. In particular, it equals to $\{\mu(x),\mu(x)+1\}$ when $x$ is non-degenerate. Another well known fact that we will use is \be\label{filter local}\mathrm{HF}^{(c-\epsilon,c+\epsilon)}(H)=\mathop{\oplus}\limits_{\mathcal{A}_H(\tilde{x})=c}\mathrm{HF}(\tilde{x}),\ee
if $c$ is the only critical value in $(c-\epsilon,c+\epsilon)$ and all $1$-periodic orbits with action $c$ are isolated.

Given a closed Reeb orbit $x$, for every admissible Hamiltonian $H$ with the slope greater than the period of $x$, there is a unique $1$-periodic orbit $\tilde{x}$. It is simpler in the local case that the local Floer homology of each such $1$-periodic orbit is the same, see Section 2.3 of \cite{F21}. Therefore, we can just define the local symplectic homology $\mathrm{SH}(x)$ of the closed Reeb orbit $x$ to be $\mathrm{HF}(\tilde{x})$. The equivariant local symplectic homology $\mathrm{SH}^{S^1}(x)$ can be defined in a similar way, starting from the equivariant Floer homology of the Hamiltonian $H\in\mathcal{H}$, see \cite{F21,GG20} for details.

Now we list two facts about $\mathrm{SH}^{S^1,I}(W)$ and $\mathrm{SH}^{S^1}(x)$ with rational coefficients. Given a non-degenerate closed Reeb orbit, we call it {\it good} if the parity of its index is the same as its underlying simple closed orbit and {\it bad} otherwise. For the filtered equivariant symplectic homology, we have the following fact.

\begin{pro}\label{good orbit complex}{\rm (c.f. Proposition 3.3 of \cite{GG20})} Assume the contact form is non-degenerate and $I\subset(0,\infty)$, then $\mathrm{SH}^{S^1,I}(W)$ can be seen as the homology of a complex generated by good closed Reeb orbits with period in $I$, graded by the Conley-Zehnder index.
\end{pro}

The local equivariant symplectic homology has been related with local symplectic homology groups by the following (cf. \cite{CGG25})
\begin{equation}\label{local-splitting}
\mathrm{SH}(x) = \mathrm{SH}^{S^1}(x) \oplus \mathrm{SH}^{S^1}(x)[-1].
\end{equation}

\subsection{Symplectically degenerate maximum }

An isolated closed Reeb orbit $x$ is called a {\it symplectically degenerate maximum} (or SDM) if $\hat{\mu}(x)\in2\Z$ and $\mathrm{SH}^{S^1}_{\hat{\mu}(x)+n-1}(x)\neq0$. Note that this definition is irrelevant to the trivialization of contact structure, although $\hat{\mu}(x)$ differs by an even number for different trivializations. We call an iteration $x^k$ {\it admissible} if $\nu(x^k)=\nu(x)$. In the following proposition, we collect some results from \cite{GHHM13}, with the knowledge that for a simple closed orbit $x$, $\mathrm{SH}^{S^1}_*(x)$ is isomorphic to the local contact homology $HC_*(x)$ with rational coefficients, up to a shift of degree.

\begin{pro}\label{SDM} {\rm(cf. Theorem 2 and Proposition 3 of \cite{GHHM13})}
 For a simple and isolated closed Reeb orbit $x$, it is an SDM if and only if $\mathrm{SH}^{S^1}_{\hat{\mu}(x^{k_i})+n-1}(x^{k_i})\neq0$ for some sequence of admissible iterations $k_i\to\infty$. Moreover, the existence of a simple SDM implies the existence of infinitely many closed Reeb orbits.
\end{pro}

\setcounter{figure}{0}
\setcounter{equation}{0}
\section{Maslov-type index for symplectic paths}
In this section, we will introduce some basic facts about the iterated Conley-Zehnder index, also called the Maslov-type index (see \cite{L02}). Actually, Maslov-type index has given a normalization of the Conley-Zehnder index where the index of a small and non-degenerate quadratic Hamiltonian $Q$ on $\R^{2d}$ equals to $\frac{1}{2}\mathrm{sgn}Q$.

The $\diamond$-sum (direct sum) of any two real matrices is defined by
$$ \left(\begin{array}{lll}A_1 & B_1\\ C_1 & D_1 \end{array}\right)_{2i\times 2i}\diamond
      \left(\begin{array}{lll}A_2 & B_2\\ C_2 & D_2 \end{array}\right)_{2j\times 2j}
=\left(\begin{array}{llll}A_1 & 0 & B_1 & 0 \\
                                   0 & A_2 & 0& B_2\\
                                   C_1 & 0 & D_1 & 0 \\
                                   0 & C_2 & 0 & D_2\end{array}\right). $$
For every $M\in\Sp(2d)$, the homotopy set $\Omega(M)$ of $M$ in $\Sp(2d)$ is defined by
$$ \Om(M)=\{N\in\Sp(2d)\,|\ \nu_{\om}(N)=\nu_{\om}(M),\, \forall\ |\om|=1\}, $$
where $\nu_{\om}(M)\equiv\dim_{\C}\ker_{\C}(M-\om I)$.
The component $\Om^0(M)$ of $M$ in $\Sp(2d)$ is defined by the path connected component of $\Om(M)$ containing $M$.

\medskip

For every continuous symplectic path $\Phi\in\mathcal{P}_\tau(2d)\equiv\{\ga\in C([0,\tau],{\rm Sp}(2d))\ |\ \ga(0)=I_{2d}\}$, we extend
$\Phi(t)$ to $t\in [0,m\tau]$ for every $m\in\N$ by
\bea \Phi^m(t)=\Phi(t-j\tau)\Phi(\tau)^j \quad \forall\;j\tau\le t\le (j+1)\tau,\ j=0, 1, \ldots, m-1. \lb{3.9}\eea
When $\Phi$ is non-degenerate, i.e., $1$ is not the eigenvalue of $\Phi(\tau)$, denote $\mu(\Phi)$ by its Maslov-type index. In the general case, the upper and lower Maslov-type index of $\Phi^m$ are considered as follow \be\label{upper index}\mu_+(\Phi^m)=\sup\{\mu(\Psi)|\Psi\in\mathcal{P}_\tau^*(2d)\ \mbox{is}\ \mbox{sufficiently}\ C^0\ \mbox{close}\ \mbox{to}\ \Phi^m\},\ee
\be\label{lower index}\mu_-(\Phi^m)=\inf\{\mu(\Psi)|\Psi\in\mathcal{P}_\tau^*(2d)\ \mbox{is}\ \mbox{sufficiently}\ C^0\ \mbox{close}\ \mbox{to}\ \Phi^m\},\ee
where $\mathcal{P}_\tau^*(2d)$ is the set of non-degenerate symplectic paths.
Their difference is bounded, \be\mu_+(\Phi)-\mu_-(\Phi)=\nu_1(\Phi(\tau)),\lb{pm_difference}\ee where $\nu_1(\Phi(\tau))$ is called the nullity of $\Phi$ and we omit its subscript later, see Theorem 6.1.8 in \cite{L02}.
The mean index is defined by $\hat{\mu}(\Phi)=\lim_{m\to\infty}\frac{\mu_-(\Phi^m)}{m}$. 

\medskip
\begin{thm}\label{homotopic normal form}{\rm(cf. Theorem 1.8.10 and Lemma 2.3.5 of \cite{L02})} For every $P\in\Sp(2d)$, there exists a continuous path $f\in\Om^0(P)$ connecting $P$ and the so-called basic normal form decomposition as follows
\bea f(1)
&=& N_1(1,1)^{\dm p_-}\,\dm\,I_{2p_0}\,\dm\,N_1(1,-1)^{\dm p_+}\nn\\
  &&\dm\,N_1(-1,1)^{\dm q_-}\,\dm\,(-I_{2q_0})\,\dm\,N_1(-1,-1)^{\dm q_+} \nn\\
&&\dm\,N_2(e^{\sqrt{-1}\aa_{1}},A_{1})\,\dm\,\cdots\,\dm\,N_2(e^{\sqrt{-1}\aa_{r_{\ast}}},A_{r_{\ast}})
  \dm\,N_2(e^{\sqrt{-1}\bb_{1}},B_{1})\,\dm\,\cdots\,\nn\\
&&\dm\,N_2(e^{\sqrt{-1}\bb_{r_{0}}},B_{r_{0}})\dm\,R(\th_1)\,\dm\,\cdots\,\dm\,R(\th_r)\dm\,H(2)^{\dm h},\lb{basic normal}\eea
where $\frac{\alpha_{j}}{2\pi},\frac{\beta_{j}}{2\pi},\frac{\th_{j}}{2\pi}\in(0,1/2)\cup(1/2,1)$, $ 2\times 2$ matrix $N_1(\pm1, b),R(\th)$ and $4\times4$ matrix $N_2(e^{\sqrt{-1}\varphi},B)$ are basic
normal forms introduced in \cite{L00}, with eigenvalue $\pm1, e^{\pm\sqrt{-1}\th}, e^{\pm\sqrt{-1}\varphi}$, respectively. In particular, $R(\th)$ is the rotation matrix $\left(\begin{array}{ll}\cos\th & -\sin\th \\
                           \sin\th & \cos\th \end{array}\right).$ Counting the dimension, we have
\be p_- + p_0 + p_+ \le d.\lb{sum_n-1} \ee
\end{thm}
\begin{thm}\label{precise_iteration}{\rm(cf. Theorem 8.3.1 of \cite{L02})}
For a symplectic path $\Phi$ ending at $P$ with the decomposition (\ref{basic normal}), we have
\bea \mu_-(\Phi^m)
&=& m(\mu_-(\Phi)+p_-+p_0-r ) + 2\sum_{j=1}^r\left\lceil\frac{m\th_j}{2\pi}\right\rceil - r   \nn\\
&&  - p_- - p_0 - {{1+(-1)^m}\over 2}(q_0+q_+) \nn
              + 2\sum_{j=1}^{r_{\ast}}\vf\left(\frac{m\aa_j}{2\pi}\right) - 2r_{\ast},\lb{iteration formula}\\
\eea
where $\lceil x\rceil=\min\{y\in\Z|y\ge x\}$, $\vf(a)=0$ if $a\in\Z$ and $\vf(a)=1$ if $a\notin\Z$.
\end{thm}

It is well known that \be |\mu_-(\Phi)-\hat{\mu}(\Phi)|\le d, \lb{mu-hat{mu}}\ee
where the inequality is strict when $P$ has an eigenvalue other than 1. This can be easily shown by Theorem \ref{precise_iteration}, for example.

\begin{pro}\label{large index growth}{\rm(cf. Theorem 2.2 of \cite{LZ02})} There holds
\be\mu_-(\Phi)-d\le\mu_-(\Phi)-\frac{e(P)}{2}\le\mu_-(\Phi^{m+1})-\mu_+(\Phi^m),\ee where $e(P)$ denotes the total algebraic multiplicity of eigenvalues of $P$ on the unit circle.
\end{pro}

\medskip

Let $S^{\pm}: \Sp(2d)\times S^1\rightarrow \N\cup\{0\}$ be the splitting number as in \cite{L02}. For a given symplectic matrix $M$ and $\omega\in S^1$, denote its splitting number by $S^{\pm}_M(\omega)$. The following properties about the splitting number can be found in Section 9.1 of \cite{L02}.

\begin{pro}\label{splitting number} {\rm(i)} $S^{\pm}_M(\omega)=S^{\pm}_N(\omega)$ for every $N$ connected to $M$ in $\Omega^0(M)$;

{\rm(ii)} $S^{\pm}_{M\dm N}(\omega)=S^{\pm}_M(\omega)+S^{\pm}_N(\omega)$;

{\rm(iii)}$S^{\pm}_M(\omega)=0, \forall \omega\notin\sigma(M)$, where $\sigma(M)$ is the spectrum of $M$;

{\rm(iv)} $S^+_M(\omega)=S^-_M(\bar{\omega})$, where $\bar{\omega}$ denotes the complex conjugate of $\omega\in S^1$;

{\rm(v)} $S^{\pm}_{N_1(1, b)}(1)=1$ if $b=0,1$, and $S^{\pm}_{N_1(1, b)}(1)=0$ if $b=-1$;

{\rm(vi)} $(S^+_{R(\theta)}(e^{\sqrt{-1}\theta}),S^-_{R(\theta)}(e^{\sqrt{-1}\theta}))=(0,1)$, $S^{\pm}_{N_2(e^{\sqrt{-1}\theta},B)}(e^{\sqrt{-1}\theta})=1$.
\end{pro}
\medskip

Now we are going to state a powerful tool in Maslov-type index theory, called {\it Common Index Jump Theorem} from \cite{LZ02}, which has an enhanced version and a generalized version (cf. \cite{DLW16, DLLW24}), also see \cite{CGG25, GG20} where it is called {\it Index Recurrence Theorem}. Here, we list the results we need in these versions. Let $C(M)=\sum_{0<\theta<2\pi} S^-_M(e^{\sqrt{-1}\th})$.

\begin{thm}\label{CIJT}{\rm(cf. \cite{DLW16}, \cite{LZ02})}
Let $\Phi_i\in\mathcal{P}_{\tau_i}(2d)$ for $i=1,\cdots,q$ be a finite collection of symplectic paths with positive mean
indices $\hat{\mu}(\Phi_i)$. Let $M_i=\Phi_i(\tau_i)$ and $\delta>0$ be sufficiently small. We extend $\Phi_i$ to $[0,+\infty)$ by (\ref{3.9}) inductively. Then for any fixed $\bar{m}\in \N$,
 there exist infinitely many $(q+1)$-tuples
$(N, m_1,\cdots,m_q) \in \N^{q+1}$ such that the following hold for all $1\le i\le q$ and $1\le m\le \bar{m}$,
\bea
 \nu(\Phi_i)&=& \nu(\Phi_i^{2m_i\pm m}),\lb{v2m_i pm m}   \\
\mu_-(\Phi_i^{2m_i+m}) &=& 2N+\mu_-(\Phi_i^m),                         \lb{2m_i+1}\\
\mu_+(\Phi_i^{2m_i-m}) &=&  2N-\mu_-(\Phi_i^m)-2(S^+_{M_i}(1)+Q_i(m))+\nu(\Phi_i^m),  \lb{2m_i-1}\\
\mu_-(\Phi_i^{2m_i}) &=& 2N-(S^+_{M_i}(1)+C(M_i)-2\Delta_i), \lb{2m_i}
\eea
where $\Delta_i=\sum_{0<\{\frac{m_i\theta}{\pi}\}<\delta}S^-_{M_i}(e^{\sqrt{-1}\th})$, $Q_i(m) = \sum_{e^{\sqrt{-1}\th}\in\sg(M_i),\atop \{\frac{m_i\th}{\pi}\}
                   = \{\frac{m\th}{2\pi}\}=0}S^-_{M_i}(e^{\sqrt{-1}\th})$, and $\{a\}=a-\max\{y\in\Z|y\le a\}$. 
\end{thm}
\begin{rem}\label{arbitrarily large tuple}
Actually, it follows from the proof that the $(q+1)$-tuple can be chosen such that $N$ and $m_i$ are larger than arbitrary real number.
\end{rem}
%
%
%
%

\setcounter{figure}{0}
\setcounter{equation}{0}
\section{Proofs of Main Theorems}\label{Proof}
In this section, we emphasize that the symplectic path is in ${\rm Sp}(2n-2)$ when applying the Maslov-type index theory to Reeb dynamics on star-shaped hypersurfaces in $\R^{2n}$.
\subsection{Proof of Theorem \ref{weak condition non-hyperbolic}}

We argue by contradiction. Changing the sign of the contact form if necessary, we can make the following assumption.
\medskip

{\bf (F1)} {\it There are finitely many simple closed Reeb orbits. Those with positive mean index are denote by $x_1,x_2,\cdots,x_r$ and $x_1$ is hyperbolic.}

\medskip

As in \cite{CGGM26}, fix a Hamiltonian $H\in\mathcal{H}$ with sufficiently large slope and satisfying the requirement of Theorem 4.1(Crossing Energy Theorem) in that paper. We will analyze the degree distance or action difference from $\hat{x}_1^{2m_1}$ to other closed Reeb orbits $x_i^j$ with period less than $slope(2m_1H)$. Except in the case of $-\bar{m}\le j-2m_i<0$, the degree distance and the action difference can be analyzed as in \cite{CGGM26}.

However, in this case, we will only prove that the distance between the degree of $\hat{x}_1^{2m_1}$ and ${\rm supp}\mathrm{HF}(\tilde{x_i}^j)$, not the larger set $[\mu_-(x_i^{j}),\mu_+(x_i^{j})+1]$ as in \cite{CGGM26}, is greater than or equal to 2. Here for a $1$-periodic orbit $\gamma$ of $H$, we denote by $\gamma^k$ the orbit $\gamma(kt)$, which is a $1$-periodic orbit of $kH$.
More precisely, we will prove
\begin{lem}\label{crucial lemma2} Take a closed Reeb orbit $x_i^j$ with period less than $slope(2m_1H)$.

{\rm(i)} If $j=2m_i$, then either $$|\mathcal{A}_{2m_1H}(\tilde{x_i}^j)-\mathcal{A}_{2m_1H}(\tilde{x_1}^{2m_1})|<\sigma$$ or $$|\mathcal{A}_{2m_1H}(\tilde{x_i}^j)-\mathcal{A}_{2m_1H}(\tilde{x_1}^{2m_1})|>C,$$
where $\sigma$ and $C$ are positive numbers determined in Theorem 4.1 and 4.3 of \cite{CGGM26}.

{\rm(ii)} If $|j-2m_i|>\bar{m}$,
then the distance from $ \mu(\hat{x}_1^{2m_1})$ to $[\mu_-(x_i^j),\mu_+(x_i^j)+1]$ is not less than 2.

{\rm(iii)} If $0<j-2m_i\le\bar{m}$, then the distance from $ \mu(\hat{x}_1^{2m_1})$ to $[\mu_-(x_i^j),\mu_+(x_i^j)+1]$ is not less than 2.

{\rm(iv)} If $-\bar{m}\le j-2m_i<0$, then the distance from $ \mu(\hat{x}_1^{2m_1})$ to ${\rm supp}\mathrm{HF}(\tilde{x_i}^{j})$ is not less than 2.
\end{lem}

As was suggested in \cite{CGGM26}, Lemma \ref{crucial lemma2} is also enough to get a contradiction. Following the ideas provided by Erman \c{C}ineli, we give some details here. Set $a=\mathcal{A}_{2m_1H}(\tilde{x_1}^{2m_1})$ and fix a sufficiently small $\epsilon>0$ such that $a$ is the only critical value of $\mathcal{A}_{2m_1H} $ in $(a-\epsilon,a+\epsilon)$, then $\hat{x}_1^{2m_1}$ is a non-trivial cycle in $\mathrm{CF}^{(a-\epsilon,a+\epsilon)}(2m_1H)$ since $n(\hat{x}_1^{2m_1},\check{x}_1^{2m_1})=0$ and other closed orbits connected to $\check{x}_1^{2m_1}$ must have action outside $(a-\epsilon,a+\epsilon)$. We divide the argument into two cases:

\medskip

{\bf Case 1}: {\it $\partial \hat{x}_1^{2m_1}$ is a nontrivial cycle in $\mathrm{CF}^{a-\epsilon}(2m_1H)$.}

\medskip

Set $b=\inf\{c<a|[\partial \hat{x}_1^{2m_1}]\in \im(\mathrm{HF}^c(2m_1H)\to \mathrm{HF}^{a-\epsilon}(2m_1H))\}$.
 By Theorem 4.1 of \cite{CGGM26}, each closed orbit entering into $\partial \hat{x}_1^{2m_1}$ has action not more than $a-\sigma$, so $b\le a-\sigma$. By Theorem 4.3 of \cite{CGGM26}, the natural inclusion map from $\mathrm{HF}^c(2m_1H)$ to $\mathrm{HF}^{a-\epsilon'}(2m_1H)$ must be zero for some $0<\epsilon'<\epsilon$, when $a-c>C$, so is the map to $\mathrm{HF}^{a-\epsilon}(2m_1H)$ and $b\ge a-C$. In particular, $b$ is finite and $[\partial \hat{x}_1^{2m_1}]\notin \im(\mathrm{HF}_{\mu(\hat{x}_1^{2m_1})-1}^{b-\epsilon_1}(2m_1H)\to \mathrm{HF}_{\mu(\hat{x}_1^{2m_1})-1}^{b+\epsilon_1}(2m_1H)) $. By (\ref{long exact sequence}),
$\mathrm{HF}_{\mu(\hat{x}_1^{2m_1})-1}^{(b-\epsilon_1,b+\epsilon_1)}(2m_1H)\neq0$. Hence, by (\ref{support range}) and (\ref{filter local}),
 there exists a Hamiltonian 1-periodic orbit $y$ with action in $[a-C,a-\sigma]$ and $\mu(\hat{x}_1^{2m_1})-1\in \mathrm{supp}\mathrm{HF}(y)$. When $m_i$ is sufficiently large, $y$ must be a non-constant $1$-periodic orbit of $2m_1H$ with positive mean index. By Lemma \ref{crucial lemma2}, such closed orbit can not exist whatever its iteration number is, a contradiction!

\medskip

{\bf Case 2}: {\it $\partial \hat{x}_1^{2m_1}$ is a boundary in $\mathrm{CF}^{a-\epsilon}(2m_1H)$.}
\medskip

By (\ref{long exact sequence}),
there exists a cycle $\beta$ in $\mathrm{CF}^{a+\epsilon}(2m_1H)\backslash \mathrm{CF}^{a-\epsilon}(2m_1H)$ such that $\beta$ and $\hat{x}_1^{2m_1}$ are homologous in $\mathrm{CF}^{(a-\epsilon,a+\epsilon)}(2m_1H)$. Similarly, for sufficiently small $\epsilon$, a chain in $\mathrm{CF}^{(a-\epsilon,a+\epsilon)}(2m_1H)$ containing $\hat{x}_1^{2m_1}$ can not be a boundary, so $\hat{x}_1^{2m_1}$ is a summand of $\beta$ and $\beta$ is nontrivial. Set $b=\sup\{c>a|\beta \notin \ker(\mathrm{HF}^{a+\epsilon}(2m_1H)\to \mathrm{HF}^c(2m_1H))\}$. Again by Theorem 4.1 of \cite{CGGM26}, $\beta$ can not be a boundary in $\mathrm{CF}^{a+\sigma}(2m_1H)$, so $b\ge a+\sigma$. By Theorem 4.3 of \cite{CGGM26}, the natural inclusion map from $\mathrm{HF}^{a+\epsilon'}(2m_1H)$ to $\mathrm{HF}^c(2m_1H)$ must be zero for some $0<\epsilon'<\epsilon$ when $c-a>C$, so is the map from $\mathrm{HF}^{a+\epsilon}(2m_1H)$ and $b\le a+C$.  In particular, $b$ is finite and $[\beta]\in \ker(\mathrm{HF}_{\mu(\hat{x}_1^{2m_1})}^{b-\epsilon_1}(2m_1H)\to \mathrm{HF}_{\mu(\hat{x}_1^{2m_1})}^{b+\epsilon_1}(2m_1H))$. By (\ref{long exact sequence}),
$\mathrm{HF}_{\mu(\hat{x}_1^{2m_1})+1}^{(b-\epsilon_1,b+\epsilon_1)}(2m_1H)\neq0$. Hence, by (\ref{support range}) and (\ref{filter local}),
 there exists a Hamiltonian 1-periodic orbit $y$ with action in $[a+\sigma,a+C]$ and $\mu(\hat{x}_1^{2m_1})+1\in \mathrm{supp}\mathrm{HF}(y)$. As in Case 1, this is impossible due to Lemma \ref{crucial lemma2}.

Strictly speaking, the element of our complex should look like $\widehat{x_i^j}$ or $\widecheck{x_i^j}$, i.e., take the reparametrizations after iterating, but it can be seen from the definition that $\hat{x_i}^j=\widehat{x_i^j}$ and $\check{x_i}^j=\widecheck{x_i^j}$. Also, it holds $\tilde{x_i}^j=\widetilde{x_i^j}$.

Therefore, it finished the proof of Theorem 1.1 and remains to prove Lemma \ref{crucial lemma2}.\hfill\hb

\begin{proof}[Proof of Lemma \ref{crucial lemma2}]
The proof of (i) and (ii) are the same as \cite{CGGM26}. Here we only prove (iii) and (iv) below.
We can apply Theorem 3.5 to closed orbits $x_1,x_2,\cdots,x_r$ and it follows from the index condition that, for $2\le m\le \ol{m}$, there holds
\bea\mu_-(x_i^{2m_i+m})&>&\mu_-(x_i^{2m_i+1})=2N+\mu_-(x_i)\ge 2N+3,\lb{index estimate 1}\\
    \mu_(x_1^{2m_1})+1&=&2N+1,\lb{index estimate 2}\\
\mu_+(x_i^{2m_i-m})+1&<&\mu_+(x_i^{2m_i-1})+1=2N-\mu_-(x_i)-2S^+_{M_i}(1)+\nu(x_i)+1\nn\\
                                       &=&2N-\mu_-(x_i)-p_{-,i}+p_{+,i}+1\nn\\
                                       &\le&\left\{\begin{array}{ll} 2N,\quad \mbox{if}\ x_i \mathrm{\ is\ totally\ degenerate}\\
           2N-1, \quad\mbox{otherwise}  \end{array}\right. ,\lb{index estimate 3}
\eea
where $M_i$ is the linearized Poincar\'{e} map of $x_i$, the first inequalities of (\ref{index estimate 1}) and (\ref{index estimate 3}) come from Proposition \ref{large index growth}, the second inequality of (\ref{index estimate 1}) is due to $n\ge3$ and the second equality of (\ref{index estimate 3}) comes from (\ref{basic normal}) and (ii) of Proposition 3.4. For the second inequality of (\ref{index estimate 3}), the totally degenerate case comes from (\ref{sum_n-1}) and the other case comes from Theorem 1.4.1 and Case 1.8.1-1.8.2 of \cite{L02}.

The statement (iii) of Lemma 4.1 immediately follows from (\ref{index estimate 1}) and (\ref{index estimate 2}).

If $2N\notin{\rm supp}\mathrm{HF}(\tilde{x_i}^{2m_i-1})$ for sufficiently large $(N,m_1,\cdots,m_r)$, then by (\ref{support range}) we have ${\rm supp}\mathrm{HF}(\tilde{x_i}^{2m_i-1})\subset(-\infty, 2N-1]$ and the proof of (iv) is finished. Otherwise, the equality in (\ref{index estimate 3}) holds for some totally degenerate closed orbit $x_i$, i.e., $\mu_-(x_i)=n,p_{+,i}=n-1$, implying that $\hat{\mu}(x_i^{2m_i-1})=(2m_i-1)n=\mu_+(x_i^{2m_i-1})-(n-1)=2N-n$ by (\ref{pm_difference}) and (\ref{iteration formula}).
Moreover, as is stated in Section \ref{Filtered and local homology}, we have $\mathrm{SH}_{2N}(x_i^{2m_i-1})\neq0$. By (\ref{local-splitting}), there holds $$\mathrm{SH}_{2N}^{S^1}(x_i^{2m_i-1})\oplus\mathrm{SH}_{2N-1}^{S^1}(x_i^{2m_i-1})\neq0.$$

If $\mathrm{SH}_{2N}^{S^1}(x_i^{2m_i-1})\neq0$, then by (\ref{local-splitting}) again, there holds $$\mathrm{SH}_{\mu_+(x_i^{2m_i-1})+2}(x_i^{2m_i-1})=\mathrm{SH}_{2N+1}(x_i^{2m_i-1})\neq0,$$ contradicting to (\ref{support range}).

If $\mathrm{SH}_{2N-1}^{S^1}(x_i^{2m_i-1})\neq0$, then we have
 $$ \mathrm{SH}^{S^1}_{\hat{\mu}(x_i^{2m_i-1})+n-1}(x_i^{2m_i-1})=\mathrm{SH}^{S^1}_{2N-1}(x_i^{2m_i-1})\neq 0,$$
 It follows from Proposition \ref{SDM} that $x_i$ is a simple SDM and there exist infinitely many simple closed orbits, which is a contradiction to (F1), see e.g. \cite{DQ26,GG20}.
\end{proof}

\subsection{Proof of Theorem \ref{3elliptic}}
Assume that there are exactly $3$ simple closed Reeb orbits $x_1,x_2,x_3$. Note that in the strongly non-degenerate case, the diamond product (\ref{basic normal}) in Theorem \ref{precise_iteration} becomes \be R(\theta_1)\dm R(\theta_2)\dm\cdots\dm R(\theta_r)\dm N_2(e^{i\varphi_1},B_1)\dm\cdots\dm N_2(e^{i\varphi_s},B_s)\dm H,\lb{nondg normalform}\ee where $\frac{\theta_j}{\pi},\frac{\varphi_j}{\pi}\notin \Q$, and the index formula (\ref{iteration formula}) becomes
\be \mu(\Phi^m)=m(\mu_-(\Phi)-r ) + 2\sum_{j=1}^r\left\lceil\frac{m\th_j}{2\pi}\right\rceil - r. \lb{nondg formula} \ee

{\bf Claim 1. } {\it No good closed orbit has Conley-Zehnder index $1$.}

\begin{proof}
Assume that there exists simple closed characteristic $x$ such that $\mu(x^m)=1$ and $\mu(x^m)\equiv\mu(x)({\rm mod\ 2})$ for some $m\in \N$. Note that the index of a symplectic path ending at $R(\theta)$ and $N_2(\omega,B)$ is odd and even, respectively, see e.g., Theorem 8.1.7, 8.2.3, 8.2.4 of \cite{L02}. Then considering the parity of the initial index and its change after iteration, by (\ref{nondg normalform}), $P_x$ must be connected to $R(\theta)\dm H(2)$ or a hyperbolic matrix $H(2)\dm H(2)$. In the latter case, there holds $m\mu(x)=1$ and then $\mu(x)=1$, a contradiction to the assumption that no simple closed orbit has index $1$. Now we look at the former case.

In this case, there holds \be m(\mu(x)-1)+2\left\lceil\frac{m\theta}{2\pi}\right\rceil-1=1.\ee
After shifting terms, we obtain a contradiction that \be 2\le|\mu(x)-1|=\frac{-2+2\lceil\frac{m\theta}{2\pi}\rceil}{m}\le \frac{2(m-1)}{m}.\ee
\end{proof}

{\bf Claim 2. } {\it There holds $\hat{\mu}(x_1),\hat{\mu}(x_2),\hat{\mu}(x_3)>0$}.

\begin{proof}
Due to Theorem 1.6 of \cite{DLLW24}, it suffices to show that there exists no simple closed characteristic with negative mean index.

At first, by Proposition \ref{good orbit complex} and (\ref{contact sphere homology}),
it is impossible that exactly one or three simple closed characteristics have negative mean index. Then it remains to exclude the case that exactly two simple closed characteristics have negative mean index.

In this case, let us assume that $\hat{\mu}(x_1),\hat{\mu}(x_2)<0$ and $\hat{\mu}(x_3)>0$. Again, by Proposition \ref{good orbit complex} and (\ref{contact sphere homology}),
we have $2+2\N\subset \{\mu(x_3^m)\}_{m\in\N}$ and $\mu(x_3)$ is even. For each $k\in \N$, $2+2k$ is achieved by $\{\mu(x_3^m)\}_{m\in\N}$ exactly once, i.e. there exists exactly one $m\in\N$ such that $\mu(x_3^m)=2+2k$. So $\mu(x_3)\neq 0$, otherwise we have $\mu(x_3^{2m_3\pm1})=2N$ by Theorem \ref{CIJT}. Now we have $\mu(x_3)\ge 2$ by (\ref{mu-hat{mu}}) and, by (\ref{iteration formula}), $\mu(x_3^m)$ is increasing with respect to $m$.

If $\mu(x_3)> 2$, then $\mu(x_3^m)=2m+2$ or, $\mu(x_3^{2m-1})=2m+2$ and $x_3^{2m}$ is bad, which is impossible for a $4$ dimensional strongly non-degenerate symplectic path by (\ref{iteration formula}). In fact, this is obvious in the case of $r=0,1$, where $r$ appears in Theorem \ref{precise_iteration}. When $r=2$, two rotation angles must be conjugate to keep the mean index an even number, i.e., $P_{x_3}$ is connected to $R(\theta)\dm R(2\pi-\theta)$ for some $\theta/\pi \notin \Q$. Then by (\ref{iteration formula}), we have \be\mu(x_3^m)=m(4-2)+2\left\lceil\frac{m\theta}{2\pi}\right\rceil
+2\left\lceil m-\frac{m\theta}{2\pi}\right\rceil-2=4m. \nn
\ee
So there holds $\mu(x_3)= 2$.

Hence, it follows from (\ref{contact sphere homology})
and Claim 1 that in the complex in Proposition \ref{good orbit complex}, $x_3$ must be cancelled by a good closed characteristic
with index 3, whose underlying simple closed characteristic is different from $x_3$ and has positive mean index by (\ref{mu-hat{mu}}), contradicting to our case.
\end{proof}
\begin{rem}
Note that, this claim also follows from the very recent \cite{CGG26}, which states that for star-shaped hypersurfaces of arbitrary dimension, all closed characteristics have non-negative mean index when they are finitely many.
\end{rem}

{\bf Claim 3. } {\it No good closed orbit has Conley-Zehnder index $-1,0$.}

\begin{proof}
If there is a good closed characteristic $y$ with Conley-Zehnder index $-1$, then similar to the above, $P_y$ is connected to $R(\theta)\dm H(2)$ or a hyperbolic matrix $H(2)\dm H(2)$. In both cases, $\mu(y)=-2+\frac{\theta}{\pi}$ or $-1<0$ due to (\ref{nondg formula}), and so is that of the underlying simple closed characteristic. This contradicts to Claim 2.

If there is a good closed characteristic $y$ with index $0$, then by (\ref{contact sphere homology}) and Claim 1, in the complex in Proposition \ref{good orbit complex}, $y$ must be cancelled by a good closed characteristic with index -1, which is impossible as is stated in the previous paragraph.
\end{proof}
It follows from Claims 1-3 and (\ref{mu-hat{mu}}) that every good closed characteristic has index no less than 2. By Theorem 1.2 of \cite{DLLW24}, all $3$ simple closed characteristics and their iterations have even Conley-Zehnder index.
In particular, the contact form is lacunary, i.e., all closed Reeb orbits have indices with the same parity, and hence dynamically convex. By Theorem B of \cite{CGGM26}, all $3$ closed characteristics are not locally maximal and hence non-hyperbolic.
And by Proposition \ref{good orbit complex} and (\ref{contact sphere homology}),
at least two simple closed characteristics, saying, $x_1, x_2$ satisfy $$\mu(x_1^{2m_1})=2N-2,\ \mu(x_2^{2m_2})=2N+2.$$ Due to (\ref{2m_i}), there holds $$\mu(x_i^{2m_i})=2N+\sum_{0<\{\frac{m_i\theta}{\pi}\}<\delta}S^-_{M_i}(e^{\sqrt{-1}\th})-\sum_{1-\delta<\{\frac{m_i\theta}{\pi}\}<1}S^-_{M_i}(e^{\sqrt{-1}\th}).$$ Then it follows from Theorem \ref{homotopic normal form} and Proposition \ref{splitting number} that
the Poincar\'{e} map of $x_1,x_2$ are connected to $$R(\theta_1)\dm R(\theta_2)\ \mathrm{(}\theta_1/\pi,\theta_2/\pi\notin\Q,\theta_1+\theta_2\neq 2\pi\mathrm{)}$$ and the Poincar\'{e} map of $x_3$ is connected to $$R(\varphi_1)\dm R(\varphi_2)\ \mathrm{(}\varphi_1/\pi,\varphi_2/\pi\notin\Q\mathrm{)}\ \mathrm{or}\ N_2(e^{\sqrt{-1}\varphi},b)\ \mathrm{(}\varphi/\pi\notin \Q\mathrm{)}.$$ It is a standard argument to show further that $x_1, x_2$ are irrationally elliptic, for example see \cite{DQ26,W22}. \hfill\hb

\medskip

As pointed in the introduction, below we will give an irrational ellipsoid in $\R^6$ with a simple closed characteristic such that its $2m_i$-th iteration never has index $2N\pm2$, where $N$ and $m_i$ are obtained by Common Index Jump Theorem. Moreover, it does not satisfy the non-resonance condition in Corollary F of \cite{CGG25}.
\begin{exa}
Denote by $E(a_1,a_2,a_3)\subset \R^6$ the ellipsoid $$\left\{(z_1,z_2,z_3)\in\C^3\bigg{|}\frac{{|z_1|}^2}{a_1}+\frac{{|z_2|}^2}{a_2}+\frac{{|z_3|}^2}{a_3}=1\right\}.$$
Let $(a_1,a_2,a_3)=(\sqrt{2},\sqrt{3},\frac{\sqrt{6}}{\sqrt{2}+\sqrt{3}})$, then there are exactly $3$ simple closed characteristics $x_1,x_2,x_3$ with linearized Poincar\'{e} map and iteration index as follows.
$$P_1=R\left(2\pi\frac{\sqrt{2}}{\sqrt{3}}\right)^{\dm2},\ \mu(x_1^m)=4m+4\left\lceil m\frac{\sqrt{2}}{\sqrt{3}}\right\rceil-2,\ \hat{\mu}(x_1)=4+4\frac{\sqrt{2}}{\sqrt{3}},$$
$$P_2=R\left(2\pi\frac{\sqrt{3}-\sqrt{2}}{\sqrt{2}}\right)^{\dm2},\ \mu(x_2^m)=8m+4\left\lceil m\frac{\sqrt{3}-\sqrt{2}}{\sqrt{2}}\right\rceil-2,\ \hat{\mu}(x_2)=4+4\frac{\sqrt{3}}{\sqrt{2}},$$
$$P_3=R\left(2\pi\frac{\sqrt{3}}{\sqrt{2}+\sqrt{3}}\right)\dm R\left(2\pi\frac{\sqrt{2}}{\sqrt{2}+\sqrt{3}}\right),\ \mu(x_3^m)=4m,\ \hat{\mu}(x_3)=4.$$
Since there are conjugate rotation angles in $P_3$, for $(N,m_1,m_2,m_3)$ in Theorem \ref{CIJT}, there must hold $\mu(x_3^{2m_3})=2N$. And the non-resonance condition is not satisfied due to $\frac{\hat{\mu}(x_3)}{\hat{\mu}(x_1)}+\frac{\hat{\mu}(x_3)}{\hat{\mu}(x_2)}=1$.
\end{exa}

\subsection{Proof of Theorem \ref{isoinvar}}
Similar to the proof of Theorem \ref{weak condition non-hyperbolic}, we argue by contradiction. Changing the sign of the contact form if necessary, we can make the following assumption.

\medskip

{\bf (F2)} {\it There are finitely many simple closed Reeb orbits. Those with positive mean index are denoted by $x_1,x_2,\cdots,x_r$, and $x_1$ is strongly non-degenerate and locally maximal.}

\medskip

As in \cite{CGGM26}, fix a Hamiltonian $H\in\mathcal{H}$ with sufficiently large slope and satisfying the requirement of Theorem 4.1 in that paper. The key result of the proof is as follows.

\begin{lem}\label{crucial lemma} Given a tuple $(N,m_1,\cdots,m_r)$ obtained in Theorem \ref{CIJT}, either for any closed orbit $y\in{\rm CF}(2m_1H)$ with $n( y,\hat{x}_1^{2m_1})\neq0$ the energy of the Floer cylinder in the cascade connecting them goes to infinity as $ m_1\to \infty$  or the same holds for any closed orbit $y\in{\rm CF}(2m_1H)$ with $n( y,\check{x}_1^{2m_1})\neq0$.\end{lem}

As \cite{CGGM26} has shown, Theorem \ref{isoinvar} can be immediately obtained from this lemma and Theorem 4.3 in that paper.

\begin{proof} Since $n(\hat{x}_1^{2m_1},\check{x}_1^{2m_1})=0$ and the index of constant orbits and the periodic orbits with non-positive mean index are bounded from above, as in \cite{CGGM26}, it suffices to show Lemma \ref{crucial lemma2} holds for either $\hat{x}_1^{2m_1}$ or $\check{x}_1^{2m_1}$, with the support in (iv) replaced by $[\mu_-(x_i^j),\mu_+(x_i^j)+1]$ as in (ii) and (iii).

So let's focus on $x_1,x_2,\cdots,x_r$.
In the case of $j=2m_i$ and $|j-2m_i|>\bar{m}$, we can argue similarly as in \cite{CGGM26} to show that, for both $\hat{x}_1^{2m_1}$ and $\check{x}_1^{2m_1}$, the action difference is arbitrarily large or too small and the degree difference is very large, respectively. Now we focus on the case of $1\le|j-2m_i|\le\bar{m}$ to show that, for either $\hat{x}_1^{2m_1}$ or $\check{x}_1^{2m_1}$, the degree difference is greater than or equal to 2.

Indeed, it is enough to show that the degree distance to $[\mu_-(x_i^j),\mu_+(x_i^j)+1]$ is greater than or equal to 2 by (\ref{upper index}) and (\ref{lower index}). Applying Theorem 3.5 to closed orbits $x_1,x_2,\cdots,x_r$, we get by the index assumption that for $2\le m\le \ol{m}$
\be\mu_-(x_i^{2m_i+m})>\mu_-(x_i^{2m_i+1})=2N+\mu_-(x_i)\ge 2N+n+1,\lb{app_2m_i+m}\ee
\be\mu(x_1^{2m_1})=2N-(C(M_1)-2\Delta_1)\in[2N-n+1,2N+n-1]\lb{app_2m_1} \ee
and
\bea
\mu_+(x_i^{2m_i-m})<\mu_+(x_i^{2m_i-1})&=&2N-\mu_-(x_i)-2S^+_{M_i}(1)+\nu(x_i)\nn\\
                                       &=&2N-\mu_-(x_i)-2(p_{-,i}+p_{0,i})+(p_{-,i}+2p_{0,i}+p_{+,i})\nn\\
                                       &=&2N-\mu_-(x_i)-p_{-,i}+p_{+,i}\le 2N-2,\lb{app_2m_i-m}
\eea
where $M_i$ is the linearized Poincar\'{e} map of $x_i$, the first inequalities of (\ref{app_2m_i+m}) and (\ref{app_2m_i-m}) come from Proposition \ref{large index growth}, the second equality comes from (\ref{basic normal}) and Proposition \ref{splitting number} and the last inequality comes from (\ref{sum_n-1}).

Now we will further narrow the range of $\mu_-(x_i^{2m_i})$ for some tuples obtained in the common index jump theorem. This relies on a symmetric argument in the proof of common index jump theorem (cf. \cite{DLW16, HW22}).

More precisely, we can obtain from Theorem \ref{CIJT} a new tuple $(N', m_1',\cdots,m_q')$ such that similar equalities to (\ref{app_2m_i+m})-(\ref{app_2m_i-m}) hold with the two term $\Delta_i$ and $\Delta_i'$ being related by the following equality (see the equality (42) in Theorem 2.8 of \cite{HW22}) \be\Delta_i+\Delta_i'=C(M_i).\ee
In particular, \be \mu(x_1^{2m_1'})=2N'-C(M_1)+2\Delta_1'=2N'+C(M_1)-2\Delta_1.\lb{2m_1'}\ee
Therefore, it follows from (\ref{app_2m_1}) and (\ref{2m_1'}) that either
\be\mu(x_1^{2m_1})\in[2N,2N+n-1]\nn\ee or \be\mu(x_1^{2m_1'})\in[2N',2N'+n-1].\nn\ee
Without loss of generality, we assume $\mu(x_1^{2m_1})\in[2N,2N+n-1]$. Then (\ref{generator degree}) and (\ref{app_2m_i+m}) imply  \be|\check{x}_1^{2m_1}|=\mu(x_1^{2m_1})\le\mu_-(x_i^{2m_i+1})-2<\mu_-(x_i^{2m_i+m})-2.\ee
If $\mu_+(x_i^{2m_i-1})\le 2N-3$, then by (\ref{app_2m_i-m}),
\be\mu_+(x_i^{2m_i-m})+1<\mu_+(x_i^{2m_i-1})+1\le \mu(x_1^{2m_1})-2=|\check{x}_1^{2m_1}|-2,\lb{distance 2m_i-m}\ee
which implies that the degree distance from $\check{x}_1^{2m_1}$ to $[\mu_-(x_i^j),\mu_+(x_i^j)+1]$ is greater than or equal to 2 for all $1\le|j-2m_i|\le\bar{m}$. Otherwise, (\ref{distance 2m_i-m}) still holds for $\check{x}_1^{2m_1}$ when $\mu(x_1^{2m_1})\ge2N+1$.

Hence, it remains to consider the case of
\be\mu_+(x_i^{2m_i-1})+1= 2N-1=\mu(x_1^{2m_1})-1\lb{aspt},\ee
which implies that the degree distance from $\hat{x}_1^{2m_1}$ to $[\mu_-(x_i^j),\mu_+(x_i^j)+1]$ is greater than or equal to 2 for all $1\le|j-2m_i|\le\bar{m}$. To be specific, there now holds
\be|\hat{x}_1^{2m_1}|=\mu(x_1^{2m_1})+1=2N+1\le\mu_-(x_i^{2m_i+1})-2<\mu_-(x_i^{2m_i+m})-2\ee
and \be\mu_+(x_i^{2m_i-m})+1<\mu_+(x_i^{2m_i-1})+1\le|\hat{x}_1^{2m_1}|-2.\ee
This completes the proof of Lemma \ref{crucial lemma}, and thus the proof of Theorem \ref{isoinvar}.
\end{proof}

%
%

\section{Acknowledgements}
The authors thank Erman \c{C}ineli and Viktor Ginzburg for many careful explanations of \cite{CGGM26} and some valuable suggestions.
\bibliographystyle{abbrv}

\begin{thebibliography}{100}

\bibitem{BO09} F. Bourgeois, A. Oancea, Symplectic homology, autonomous Hamiltonians, and Morse-Bott moduli spaces. {\it Duke Math. J.} 146, 71-174 (2009).
\bibitem{BO13} F. Bourgeois, A. Oancea, The Gysin exact sequence for $S^1$-equivariant symplectic homology. {\it J. Topol. Anal.} 5, 361-407 (2013).
\bibitem{CGG25} E. \c{C}ineli, V. L. Ginzburg, B. Z. G\"{u}rel, Closed Orbits of Dynamically Convex Reeb Flows: Towards the HZ- and Multiplicity Conjectures. Preprint arXiv:2410.13093v4 (2026).
\bibitem{CGG26} E. \c{C}ineli, V. L. Ginzburg, B. Z. G\"{u}rel, Filtered Symplectic Homology and Closed Reeb Orbits. Preprint arXiv:2606.13372.
\bibitem{CGGM26} E. \c{C}ineli, V. L. Ginzburg, B. Z. G\"{u}rel, M. Mazzucchelli, Invariant sets and hyperbolic closed Reeb orbits {\it Adv. Math.} 484, Paper No. 110709, 57 pp (2026).
\bibitem{CHHL24}D. Cristofaro-Gardiner, U. Hryniewicz, M. Hutchings, H. Liu, Proof of Hofer-Wysocki-Zehnder's two or infinity conjecture. To appear in {\it J. Amer. Math. Soc.} (2026).
\bibitem{DLW16} H. Duan, Y. Long, W. Wang, The enhanced common index jump theorem for symplectic paths and
 non-hyperbolic closed geodesics on Finsler manifolds.  {\it Calc. Var. and PDEs.} 55, Art. 145, 28 pp (2016).
\bibitem{DLLW24} H. Duan, H. Liu, Y. Long, W. Wang, Generalized common index jump theorem with applications to closed
characteristics on star-shaped hypersurfaces and beyond. {\it J. Funct. Anal.} 286, Paper No. 110352, 41 pp (2024).
\bibitem{DLLQW25} H. Duan, H. Liu, Y. Long, Z. Qi, W. Wang, Three closed characteristics on
non-degenerate star-shaped hypersurfaces in $\R^6$. {\it J. Differ. Equ.} 446, Paper No. 113605, 30 pp (2025).
\bibitem{DLR22}H. Duan, H. Liu, W. Ren, A dichotomy result for closed characteristics on compact star-shaped hypersurfaces in $\R^{2n}$. {\it Math. Z.} 302, 743-757 (2022).
\bibitem{DQ26} H. Duan, Z. Qi, Closed Reeb orbits on contact type hypersurfaces in $T^*S^n$. Preprint 	arXiv:2603.07894.
\bibitem{DX25} H. Duan, D. Xie, Four closed characteristics on compact star-shaped hypersurfaces in $\R^8$. {\it J. Geom. Anal. }35, Paper No. 289, 14 pp (2025).
\bibitem{F21} E. Fender, Two Perspectives on the Local Symplectic Homology of closed Reeb Orbits, PhD Thesis, UCSC, 2021.
\bibitem{F89a} A. Floer, Symplectic fixed points and holomorphic spheres. {\it Comm. Math. Phys.} 120, 575-611 (1989).
\bibitem{F89b} A. Floer, Witten's complex and infinite-dimensional Morse theory. {\it J. Differential Geom.} 30, 207-221 (1989).
\bibitem{G10} V. L. Ginzburg, The Conley conjecture. {\it Ann. Math.} 172, 1127-1180 (2010).
\bibitem{GG20} V. L. Ginzburg, B. Z. G\"{u}rel, Lusternik-Schnirelmann theory and closed Reeb orbits. {\it Math. Z.} 295, 515-582 (2020).
\bibitem{GHHM13} V. L. Ginzburg, D. Hein, U. L. Hryniewicz, L. Macarini, Closed Reeb orbits on the sphere and symplectically degenerate maxima. {\it Acta Math. Vietnam.} 38, 55-78 (2013).
\bibitem{HW22} M. Hamid, W. Wang, A symmetric property in the enhanced common index jump theorem with
applications to the closed geodesic problem. {\it Discrete Contin. Dyn. Syst.} 42, 1933-1948 (2022).
\bibitem{HWZ03} H. Hofer, K. Wysocki, E. Zehnder, Finite energy foliations of
tight three-spheres and Hamiltonian dynamics. {\it Ann. Math.} 157, 125-255 (2003).
%
\bibitem{LLW25} Xiaorui Li, Hui Liu, Wei Wang, Two irrationally elliptic closed orbits of Reeb flows on the boundary of star-shaped domain in $\R^{2n}$. Preprint arXiv:2510.06597.
%
\bibitem{L00} Y. Long,  Precise iteration formulae of the Maslov-type index
theory and ellipticity of closed characteristics.  {\it Adv. Math.} 154, 76-131 (2000).
\bibitem{L02} Y. Long,  Index Theory for Symplectic Paths with Applications.
Progress in Math. 207, Birkh\"auser (2002).
%
\bibitem{LZ02} Y. Long, C. Zhu, Closed charateristics on compact convex hypersurfaces
in $\R^{2n}$. {\it Ann. Math.} 155, 317-368 (2002).

\bibitem{W22} W. Wang, Irrationally elliptic closed characteristics on compact convex hypersurfaces in $\R^{2n}$. {\it
 J. Funct. Anal.} 282, Paper No. 109269, 29 pp (2022).
\bibitem{WHL07} W. Wang, X. Hu, Y. Long, Resonance identity, stability and multiplicity of closed
characteristics on compact convex hypersurfaces. {\it Duke Math. J.} 139, 411-462 (2007).
\end{thebibliography}

\end{document}